\documentclass[12pt]{amsart}
\usepackage{amscd,amsmath,amsthm,amssymb,graphics}
\usepackage{slashbox}
\usepackage{pgf,tikz, pifont}
\usetikzlibrary{arrows}
\usepackage[all]{xy}
\usepackage{lineno}
\def\frk{\frak}               % font for "Fraktur"

\def\Phi{{\frk n}}
\def\Phi{{\frk N}}
\def\opn#1#2{\def#1{\operatorname{#2}}} % to make operators
\opn\chara{char} \opn\length{\ell} \opn\pd{pd} \opn\rk{rk}
\opn\projdim{proj\,dim} \opn\injdim{inj\,dim} \opn\rank{rank}
\opn\depth{depth} \opn\grade{grade} \opn\height{height}
\opn\embdim{emb\,dim} \opn\codim{codim}

\opn\Tr{Tr} \opn\bigrank{big\,rank}
\opn\superheight{superheight}\opn\lcm{lcm}
\opn\trdeg{tr\,deg}%\emph{
\opn\reg{reg} \opn\lreg{lreg} \opn\ini{in} \opn\lpd{lpd}
\opn\size{size}\opn\bigsize{bigsize}
\opn\cosize{cosize}\opn\bigcosize{bigcosize}
\opn\sdepth{sdepth}\opn\sreg{sreg}
\opn\link{link}\opn\fdepth{fdepth}
\opn\index{index}
\opn\index{index}
\opn\indeg{indeg}
\opn\N{N}
\opn\mult{mult}
\opn\SSC{SSC}
\opn\SC{SC}
\opn\lk{lk}
\opn\HS{HS}
\opn\div{div} \opn\Div{Div} \opn\cl{cl} \opn\Cl{Cl}
\opn\Spec{Spec} \opn\Supp{Supp} \opn\supp{supp} \opn\Sing{Sing}
\opn\Ass{Ass} \opn\Min{Min}\opn\Mon{Mon} \opn\dstab{dstab} \opn\astab{astab}
\opn\Syz{Syz}
\opn\reg{reg}
\opn\Ann{Ann} \opn\Rad{Rad} \opn\Soc{Soc}
\opn\Im{Im} \opn\Ker{Ker} \opn\Coker{Coker} \opn\Am{Am}
\opn\Hom{Hom} \opn\Tor{Tor} \opn\Ext{Ext} \opn\End{End}\opn\Der{Der}
\opn\Aut{Aut} \opn\id{id}

\opn\nat{nat}
\opn\pff{pf}%   \pf exists already
\opn\Pf{Pf} \opn\GL{GL} \opn\SL{SL} \opn\mod{mod} \opn\ord{ord}
\opn\Gin{Gin} \opn\Hilb{Hilb}\opn\sort{sort}
\opn\initial{init}
\opn\ende{end}
\opn\height{height}
\opn\type{type}
\opn\aff{aff} \opn\con{conv} \opn\relint{relint} \opn\st{st}
\opn\lk{lk} \opn\cn{cn} \opn\core{core} \opn\vol{vol}
\opn\link{link} \opn\Link{Link}\opn\lex{lex}
\opn\gr{gr}

\def\pot#1#2{#1[\kern-0.28ex[#2]\kern-0.28ex]}

\opn\dirlim{\underrightarrow{\lim}}
\opn\inivlim{\underleftarrow{\lim}}
\let\To=\longrightarrow
\def\Implies{\ifmmode\Longrightarrow \else
        \unskip${}\Longrightarrow{}$\ignorespaces\fi}
\def\implies{\ifmmode\Rightarrow \else
        \unskip${}\Rightarrow{}$\ignorespaces\fi}
\def\iff{\ifmmode\Longleftrightarrow \else
        \unskip${}\Longleftrightarrow{}$\ignorespaces\fi}

\let\:=\colon
\newtheorem{Theorem}{Theorem}[section]
 \newtheorem{Lemma}[Theorem]{Lemma}
 \newtheorem{Corollary}[Theorem]{Corollary}
 \newtheorem{Proposition}[Theorem]{Proposition}
 \newtheorem{Remark}[Theorem]{Remark}

 \newtheorem{Definition}[Theorem]{Definition}

\newtheorem{Notation}[Theorem]{Notation}

\let\epsilon\varepsilon
\let\kappa=\varkappa
\def\qed{\ifhmode\textqed\fi
      \ifmmode\ifinner\quad\qedsymbol\else\dispqed\fi\fi}
\def\textqed{\unskip\nobreak\penalty50
       \hskip2em\hbox{}\nobreak\hfil\qedsymbol
       \parfillskip=0pt \finalhyphendemerits=0}
\def\dispqed{\rlap{\qquad\qedsymbol}}

\opn\dis{dis}
\def\pnt{{\raise0.5mm\hbox{\large\bf.}}}

\opn\Lex{Lex}
\def\F{{\mathcal F}}
\opn\Spec{Spec} \opn\Supp{Supp} \opn\supp{supp}
 \opn\Ass{Ass}
 \opn\p{Ass}
 \opn\min{min}
 \opn\max{max}
 \opn\MIN{Min}
 \opn\p{\mathfrak{p}}
\opn\Deg{Deg}
\begin{document}

 \title{Multiplicity  of powers of path ideals of a line graph}

\author {Jiawen Shan, Zexin Wang, Dancheng Lu}
\address{School of Mathematical Science,  Soochow University, 215006 Suzhou, P.R.China}
\email{ShanJiawen826@outlook.com}

\address{School  of Mathematical Science, Soochow University, 215006 Suzhou, P.R.China}
\email{zexinwang6@outlook.com}

\address{School of Mathematical Science,  Soochow University, 215006 Suzhou, P.R.China}
\email{ludancheng@suda.edu.cn}

\dedicatory{Dedicated to the memory of Professor J\"urgen Herzog }

 \begin{abstract} Let $S=\mathbb{K}[x_1,\ldots,x_n]$ and let $I$ be the $t$-path ideal  of the line graph  $L_n$ with $n$-vertices.  It is shown  that the set of associated prime ideals of $I^s$ is equal to the set of minimal prime ideals of $I$ for all $s\geq 1$, and we provide an explicit description of these prime ideals.  Additionally, as the main contribution of this paper, we derive an   explicit formula for the multiplicity of  $S/I^s$ for all $s\geq 1$, revealing that it is a polynomial in $s$ from the  beginning.
 \end{abstract}


 \maketitle

\section{Introduction}

 Let $G$ be a finite simple graph with  vertex set $V=\{x_1, \ldots, x_n\}$ and  edge set $E$. The ideal generated by all quadratic monomials $x_ix_j$ such that $\{x_i,x_j\}\in E$ is called the {\it edge ideal} of $G$, denoted by $I(G)$. The edge ideal was introduced and studied by Villarreal in \cite{Vil90}. Since then, the study of edge ideals of simple graphs, particularly the properties of their powers, has become a vibrant  topic in combinatorial commutative algebra. For more detailed  information in this direction, we refer to the exposition paper \cite{BBH17} and the references therein.

 In 1999, Conca and De Negri introduced in their work \cite{CD99} the concept of a $t$-path ideal, which serves as a generalization of an edge ideal.
 For an integer $2\leq t\leq n$, a {\it  $t$-path} of the graph $G$ is a sequence of $(t-1)$ distinct edges $\{x_{i_1}, x_{i_2}\}, \{x_{i_2}, x_{i_3}\},\cdots,$ $ \{x_{i_{t-1}}, x_{i_t}\}$ of $G$ such that the vertices $x_{i_1}, x_{i_2}, \ldots, x_{i_{t}}$ are pairwise distinct. Such a path is also denoted by $\{x_{i_1}, \ldots, x_{i_t}\}$  for short. The {\it  $t$-path ideal} $I_t(G)$ associated to $G$ is defined as  the square-free monomial ideal $$I_t(G)=(x_{i_1}\cdots x_{i_t}\:\; \{x_{i_1}, \ldots, x_{i_t}\} \textrm{~~is~~ a}~~t\textrm{-path~~ of}~~ G)$$ in the polynomial ring $\mathbb{K}[x_1,\ldots,x_n]$. In the recent years, some algebraic properties of path ideals have been investigated extensively, see for instance \cite{AF18,BHO11,BHO12,HV10,KO14,KS14}.

However, very  little is known about the powers of $t$-path ideals for $t\geq 3$. In this paper, we will specifically concentrate on the powers of path ideals of {\it line graphs}.  A {\it line graph} (or {\it path graph}) of length $(n-1)$, denote by $L_n$, is a graph with vertex set $\{x_1, \ldots, x_n\}$ and  edge set  $\{\{x_j,x_{j+1}\}\:\; j=1,\ldots, n-1\}$.  Recently, ${\rm B\breve{a}l\breve{a}nescu}$-${\rm Cimpoea\textrm{\c{s}}}$ \cite{BC23} derived an explicit formula for the depth of powers of path ideals of line graphs. In \cite{SL23}, the regularity  function of powers of $I_t(L_n)$ is presented, revealing it is linear from the beginning.

 In this paper, we will study algebraic properties of powers of $I_t(L_n)$ from   two different angles: {\it associated prime ideals} of $I_t(L_n)^s$ and the {\it  multiplicities} of $I_t(L_n)^s$ for all $s\geq 1$.

 \subsection{Associated prime ideals}Let $S$ be a Noetherian ring and $M$ a finitely generated  non-zero $S$-module. A prime ideal $\p$ of $S$ is called an {\it associated prime ideal} of $M$ if there exists $x\in M$ such that $\p=(0:_Sx)$. The set of all associated prime ideal of $M$ is denoted by $\mathrm{Ass}_S(M)$. It is known that $\mathrm{Ass}_S(M)$ is a non-empty finite set. If $I$ is an ideal of $S$, we denote $\mathrm{Ass}_S(S/I)$ by $\mathrm{ass}(I)$. In 1979, Brodmann \cite{Bro79} proved the remarkable result that {\it there exists an integer  $s_0$ such that $\mathrm{ass}(I^s)= \mathrm{ass}(I^{s_0})$ for all $s\geq s_0$.} This result serves as a good example of the concept in commutative algebra that ideals exhibit favorable behavior  asymptotically.
Naturally, Brodmann's theorem also inspires several questions, such as determining the value of $s_0$ and determine the prime ideals in $\mathrm{ass}^{\infty}(I):=\mathrm{ass}(I^{s_0})$.  We refer to \cite{CHHT20} and the references therein for a comprehensive introduction to this topic.  If $s_0=1$, then  $I$ is said to be normally torsion-free.  Formally, we give the following definition following \cite{HH}.
\begin{Definition} \em An ideal $I$ of a Noetherian ring $S$ is said to be {\it normally torsion-free} if $\mathrm{ass}(I^s)= \mathrm{ass}(I)$ for all $s\geq 1$.
\end{Definition}
We will show that both $I_t(L_n)$ and $I_t(L_n)^{\vee}$ are normally torsion-free and determine the prime ideals in $\mathrm{ass}(I_t(L_n))$.

\subsection{Multiplicity} Let $S=\mathbb{K}[x_1,\ldots,x_n]$ be a  polynomial ring over a field $\mathbb{K}$ that is standardly graded  and let $M$ be a finitely generated graded $S$-module.
The Hilbert function $H_M$ of $M$ is defined as the function $$H_M(k):=\dim_{\mathbb{K}}M_k,$$ where $M_k$ is the degree $k$ component of $M$.
The Hilbert series $\HS(M,z)$ is defined to be $$\HS(M,z):=\sum_{k\in \mathbb{Z}}H_M(k)z^k.$$
If we assume $\dim M=d+1$, it is a celebrated result of David Hilbert  that $\HS(M,z)$ is a rational function of the following form $$\HS(M,z)=\frac{Q(M,z)}{(1-z)^{d+1}}.$$
Here, $Q(M,z)\in \mathbb{Q}[z^{-1},z]$ is  a Laurent polynomial such that $Q(M,1)\neq 0$. As  a consequence of this result, the Hilbert function $H_M$ is of polynomial type of degree $d$, meaning that there exists a polynomial $p_M(z)\in \mathbb{Q}[z]$ of degree $d$ such that $H_M(k)=p_M(k)$ for all $k\gg 0$. The polynomial $p_M(z)$ is referred to as the {\it Hilbert polynomial} of $M$.

\begin{Definition} \em Let $M$ be a finitely generated graded $S$-module of dimension $d+1$. The  Hilbert polynomial $p_M(z)$ of $M$ can be written as $$p_M(z)=\sum_{i=0}^d(-1)^ie_i(M)\binom{z+d-i}{d-i}.$$
The integer coefficients $e_i(M)$ for $i=0,\ldots,d$ are called the {\it Hilbert coefficients} of $M$. Among these, the first Hilbert coefficient $e_0(M)$  holds particular significance. It is also known as the {\it multiplicity} of $M$ and and is frequently denoted by  $\mult(M)$.
\end{Definition}
 Recall that a function $f:\mathbb{N}\rightarrow \mathbb{Q}$ is of {\it polynomial type} of degree $d$ if there exists a polynomial $p(z)\in \mathbb{Q}[z]$ of degree $d$ such that $f(k)=p(k)$ for all $k\gg 0$. Let $I$ be a graded ideal of $S$ such that $\dim M/IM=d$.
It was shown in \cite[Theorem 1.1]{HPV08} that {\it $e_i(M/I^sM)$ is of polynomial type in $s$ of degree $\leq n-d+i$ for $i=1,2,\ldots,d$}. To the best of our knowledge, there are currently no known non-trivial examples of graded ideals for which the multiplicity of their powers has been explicitly computed. In this paper, we will provide an explicit formula for the multiplicity of $S/I_t(L_n)^s$.

\subsection{Structure of this paper} In the paper, we always assume that $n,t$ are integers with $n\geq t\geq1$. The polynomial ring $\mathbb{K}[x_1,\ldots,x_n]$ will  be denoted by $S$. We use $L_n$ to denote the line graph with $n$ vertices and let $I_t(L_n)$ be the $t$-path ideal of the line graph $L_n$. This paper is organized as follows. In Section 2 we demonstrate  that both $I_t(L_n)$ and its Alexander dual $I_t(L_n)^{\vee}$ are normally torsion-free and provide a complete characterization of all the associated prime ideals of  $I_t(L_n)$.  The final section contains the main contribution of this paper, in which we present an   explicit formula for the multiplicity of  $S/I_t(L_n)^s$ for all $s\geq 1$, revealing that it is a polynomial in $s$ from the  beginning.

  We refer to the book \cite{HH} for the notation that is not explained in this paper.

\section{Associated primes of $I_t(L_n)^s$}

According to the definition of a $t$-path ideal,  the ideal $I_t(L_n)$ has a unique  minimal generating set $u_1,\ldots,u_{n-t+1}$, where $u_i=x_i\cdots x_{i+t-1}$ for $i=1,\ldots,n-t+1$.
In this section, we demonstrate  that both $I_t(L_n)$ and its Alexander dual $I_t(L_n)^{\vee}$ are normally torsion-free  and offer an explicit description of minimal prime ideals of $I_t(L_n)$. Furthermore, utilizing a novel method, we  compute the projective dimension of $S/I_t(L_n)$.

Recall that a simplicial complex on $[n]:=\{1,\ldots,n\}$ is a collection $\Delta$ of subsets of $[n]$ such that if $F\in \Delta$ and $G\subseteq F$, then $G\in \Delta$.  Any $F\in \Delta$ is a {\it face} of $\Delta$. A maximal face (with respect to inclusion) is a {\it facet} of $\Delta$. The set of all facets of $\Delta$ is denoted by $\mathcal{F}(\Delta)$. A simplicial complex $\Delta$ is denoted by $\langle F_1, F_2,\ldots, F_r\rangle$ if $ F_1, F_2,\ldots, F_r$ are all facets of $\Delta$.

For a subset $F$ of $[n]$, let $x_F$ denote the monomial $\prod_{i\in F}x_i$. Then the facet ideal of $\Delta$ is defined to be the monomial ideal $$I(\Delta)=(x_F\:\; F \mbox{  is a facet of }\Delta),$$ and the Stanley-Reisner $I_{\Delta}$ ideal of $\Delta$ is defined as $$I_{\Delta}=(x_F\:\; F \subseteq [n] \mbox{ and } F\notin \Delta).$$

   Put $F_i:=\{i,\ldots, i+t-1\}\subseteq [n]$  for $i=1,\ldots,n-t+1$ and define  $\Delta_{n,t}$ to be the simplicial complex generated by  $F_1,\ldots,F_{n-t+1}$.  It is evident  that the path ideal $I_t(L_n)$ is equal to the facet ideal $I(\Delta_{n,t})$ of $\Delta_{n,t}.$ A simplicial complex $\Delta$ is called a {\it flag complex} if its Stanley-Reisner ideals is generated by monomials of degree 2, or equivalently, all minimal non-faces of $\Delta$ have exactly two elements. Here, a {\it non-face} of $\Delta$ refers to a subset of $[n]$ which does not belong to $\Delta$. The lemma below uncovers an intriguing characteristic of $\Delta_{n,t}$, although we won't explore it further in this paper.

\begin{Lemma} \label{flag}
For all $2\leq t\leq n$, the Stanley-Reisner ideal $I_{\Delta_{n,t}}$ is generated by all quadratic monomials $x_ix_{j}$ with $j-i\geq t.$ In particular, $\Delta_{n,t}$ is a flag complex.
\end{Lemma}

\begin{proof} Let $I$ denote the ideal $(x_ix_{j}\:\; j-i\geq t).$ It is clear that if $j-i\geq t$, then $\{i, j\}$ is a non-face of $\Delta_{n,t}$. Thus, $I\subseteq I_{\Delta_{n,t}}$.

 Conversely, let $F=\{i_1,\ldots,i_k\}$ be a non-face of $\Delta_{n,t}$ with $i_1<\cdots< i_k$. Then if $i_k-i_1<t$, then $\{i_1,\ldots,i_k\}\subseteq F_{i_1}$, a contradiction. This implies $i_k-i_1\geq t$ and so $x_F\in I$. This proves  $I\supseteq I_{\Delta_{n,t}}.$\end{proof}

  A facet $F$ of a simplicial complex $\Delta$ is called a \textit{leaf}  if either $F$ is the only facet of $\Delta$, or there exists a facet $G\neq F$ such that $H\cap F\subseteq G\cap F$ for any facet $H\neq F$. A leaf $F$ of  $\Delta$ is said to be \textit{good}  if the sets $F\cap H$, where $H\in \mathcal{F}(\Delta)$ is linearly ordered. A simplicial complex $\Delta$ is a \textit{simplicial forest} if every nonempty subcollection of $\Delta$ has a leaf. A connected simplicial forest is called a \textit{simplicial tree}. Here, a {\it subcollection} of $\Delta$ refers to a simplicial complex whose set of facets is a subset of $\mathcal{F}(\Delta).$ The following observation forms the foundation of all the results in this section.

\begin{Lemma} \label{tree}
For all $0<t\leq n$, the simplicial complex $\Delta_{n,t}$ is a simplicial tree.
\end{Lemma}
\begin{proof} It is clear that  $\Delta_{n,t}$ is connected. Let $\Delta'=\langle F_{k_1},\ldots F_{k_r}\rangle$ be a subcollection of $\Delta$. Note that $F_{k_i}=\{k_i,k_i+1,\ldots, k_i+t-1\}$ for $i=1,\ldots,r$. We may assume that $r\geq 2$ and that $1\leq k_1<\cdots < k_r\leq n-t+1$. It can be easily verified that $$F_{k_1}\cap F_{k_i}=\left\{
\begin{array}{ll}
\{k_i, k_{i}+1, \ldots, k_{1}+t-1\}, & \hbox{if $k_i\leq k_{1}+t-1$;} \\
\emptyset, & \hbox{otherwise}
\end{array}
\right.$$
for $i=2,\ldots,r$. From this it follows that $$F_{k_1}\cap F_{k_2}\supset F_{k_1}\cap F_{k_3}\supset \cdots \supset F_{k_1}\cap F_{k_r}.$$
 Therefore $F_{k_1}$ is a (good) leaf of $\Delta'$.   Hence, $\Delta_{n,t}$ is a simplicial tree.
\end{proof}

We now present the  main result of this section.
\begin{Theorem}\label{main1}
For all $1\leq t\leq n$, both $I_t(L_n)$ and $I_t(L_n)^{\vee}$ are normally torsion-free.
\end{Theorem}
\begin{proof}
It was shown in \cite[Theorem 2.2 and Theorem 2.5]{HHTZ08} that if a simplicial complex $\Delta$ has no special odd cycles, then both the facet ideal $I(\Delta)$  and its Alexander dual $I(\Delta^*)=I(\Delta)^{\vee}$  are normally torsion-free.   On the other hand, by \cite[Theorem 3.2]{HHTZ08}, if  $\Delta$ is a simplicial forest, then $\Delta$ has no special cycles. In view of these facts, the result follows immediately from Lemma~\ref{tree}.
\end{proof}

We note that the result that $I_t(L_n)$ is normally torsion-free  has been  established in \cite[Proposition 3.2]{NQ24} utilizing a completely different method.

By applying \cite[Theorem 1.4.6]{HH} and \cite[Corollary 1.6]{HHTZ08} respectively, we immediately obtain the following corollaries.
\begin{Corollary}
Let $1\leq t\leq n$. Then $I_t(L_n)$ is normal, this is, $\overline{I_t(L_n)^s}=I_t(L_n)^s$ for all $s\geq 1$.
 \end{Corollary}

\begin{Corollary}
Let $1\leq t\leq n$. Then the Rees algebra $\mathcal{R}(I_t(L_n))$ is a normal Cohen-Macauly domain.
\end{Corollary}

In the following, we aim to determine all the minimal prime ideals of $I_t(L_n)$, a task that will also prove useful in the subsequent section. To do this, we introduce some additional notation.  Specifically, if $F\subseteq [n]$, we denote by $\p_F$ the prime ideal generated by $x_i$ with $i\in F$.  For convenience, we also adopt the following notation.

\begin{Notation} \label{2.6} Let $\mathbf{C}_{n,t}$ be the set of subsets $\{i_1,\ldots,i_r\}$ of $[n]$ with $1\leq i_1<\cdots<i_r\leq n$ and such that the following conditions are satisfied:
\begin{enumerate}
   \item $1\leq i_1\leq t$;
   \item $i_2> t$;
   \item $1\leq i_{j+1}-i_j\leq t$ for all $j=1,\ldots,r-1$;
   \item $i_{j+2}-i_j> t$ for all $j=1,\ldots,r-2$;
   \item $i_{r-1}< n-t+1$;
   \item $n-t+1\leq i_r\leq n$.
\end{enumerate}
\end{Notation}

Recall  that a \textit{vertex cover} of the simplicial complex $\Delta$ is a subset $C\subset [n]$ such that $C\cap F\neq\emptyset$ for all $F\in\F(\Delta)$. A vertex cover $C$ is called a \textit{minimal vertex cover}, if no proper subset of $C$ is a vertex cover. It is well-known that $\p$ is a minimal prime ideal of $I(\Delta)$ if and only if $\p=\p_C$ for some minimal vertex cover of $\Delta$.

\begin{Proposition}\label{C_{n,t}} For $1\leq t\leq n$, the set of minimal prime ideals of $I_t(L_n)$ is given by  $$\{\p_F\:\; F\in \mathbf{C}_{n,t}\}.$$
\end{Proposition}
\begin{proof} It is equivalent to showing that $\mathbf{C}_{n,t}$ consists of all minimal vertex covers of $\Delta_{n,t}$. Let $F\in \mathbf{C}_{n,t}$. We first show that $F$ is a  vertex cover of $\Delta_{n,t}$. Given a facet $F_i$ of $\Delta_{n,t}$, where $1\leq i\leq n-t+1$. If $i\in F$, then $i\in F\cap F_i$ and so $F\cap F_i\neq \emptyset$. Suppose that $i\notin F$. If $i<i_1$ then $i_1\in F_i\cap F$ and so $F_i\cap F\neq \emptyset$. If $i>i_1$, then, since $i_r\geq n-t+1$,  there exists $1\leq j\leq r-1$ such that $i_j<i<i_{j+1}$. Since $i_{j+1}-i<i_{j+1}-i_j\leq t$, it follows that $i_{j+1}\in F_i$ and thus $F\cap F_i\neq \emptyset$. This proves $F$ is a vertex cover of $\Delta_{n,t}$. Next, we show that $F$ is minimal, that is,  $F\setminus \{i_j\}$ is not a vertex cover of $\Delta_{n,t}$ for $j=1,\ldots,r$. But, this is clear by noticing that $F_1\cap (F\setminus \{i_1\})=\emptyset$, $F_{i_{j-1}+1}\cap (F\setminus \{i_j\})= \emptyset$ for $j=2,\ldots, r-1$ and $F_{n-t+1}\cap (F\setminus \{i_r\})= \emptyset$. Therefore, $F$ is a minimal vertex cover of $\Delta_{n,t}$ indeed.

Conversely, let $F$ be a minimal vertex cover of $\Delta_{n,t}$. We may write $F=\{i_1,\ldots,i_r\}$ with $i_1<\ldots<i_r$. Since $F_1\cap F\neq \emptyset$, it follows that $i_1\leq t$ and so (1) holds. Similarly, since $F_{n-t+1}\cap F\neq \emptyset$, we have $i_r\geq n-t+1$ and this proves (6).  If $i_2\leq t$, then, note that $i_1\in F_i$ implies $i_2\in F_i$ for all $i=1,\ldots,n-t+1$, it follows that $F\setminus \{i_1\}$ is also a vertex cover of $\Delta_{n,t}$, a contradiction. This establishes (2). Similarly, (5) has to hold.  If   $i_{j+1}-i_j> t$ for some $j\in \{1,\ldots,r-1\}$, then $F_{i_j+1}\cap F=\emptyset$. This is impossible and so (3) should hold. It remains to prove (4). Suppose on the contrary  that $i_{j+2}-i_j\leq t$ for some $1\leq j\leq r-2$. Then, if $i_{j+1}\in F_i$ for some $i$, we have $i_{j+1}\in \{i, i+1, \ldots, i+t-1\}$. It follows that either $i_{j+2}\leq i+t-1$ (i.e., $i_{j+2}\in F_i$) or $i_j\geq i$ (i.e., $i_j\in F_i$), for otherwise $i_{i+2}-i_j>t$. Thus, $F\setminus \{i_{j+1}\}$ is also a vertex cover of $\Delta_{n,t}$, which is a contradiction. This completes the proof.
\end{proof}

This result can  be checked using \cite[Theorem 3.1]{BS21}. Combining Theorem~\ref{main1} with Proposition~\ref{C_{n,t}}, we obtain the following result.

\begin{Corollary} Let $2\leq t\leq n$. Then $\mathrm{ass}(I_t(L_n)^s)=\{\p_F\:\; F\in \mathbf{C}_{n,t}\}$ for all $s\geq 1$.
\end{Corollary}

Given a monomial ideal $I$, we use $\Deg(I)$ to denote the maximal degree of  minimal generators of $I$. We now determine  $\Deg(I_t(L_n)^{\vee})$ for all $1\leq t\leq n$. Recall for a square-free monomial ideal  $I=I(\Delta)$, the Alexander dual of $I$ is defined to be the monomial ideal $$I^{\vee}=(x_C\:\; C \mbox{ is a vertex cover of } \Delta).$$

\begin{Corollary}\label{Deg} The Alexander dual $I_t(L_n)^{\vee}$ of $I_t(L_n)$ has the minimal generating set $G=\{x_F\:\; F\in \mathbf{C}_{n,t}\}$. Assume $n=(t+1)p+q$ with $0\leq q\leq t$. Then the maximal  degree of minimal generators  of $I_t(L_n)^{\vee}$ is given by $$\Deg(I_t(L_n)^{\vee})=\left\{
\begin{array}{ll}
2p, & \hbox{$q<t$;} \\
2p+1, & \hbox{$q=t$.}
\end{array}
\right.
$$
\end{Corollary}

\begin{proof} The first statement is immediate from the definition of the Alexander dual. To prove the second statement, we denote $\Deg(I_t(L_n)^{\vee})$ as $D$ for simplicity. Consider the following infinite sequence  $$1, t+1, t+2, 2t+2, 2t+3,\cdots, kt+k, kt+k+1, \cdots .$$
Let $a_m$ denote the $m$-th number of this sequence. Namely, $a_1=1, a_2=t+1, a_3=t+2$ and so on. Thus, $a_{2k}=kt+k$ and $a_{2k+1}=kt+k+1$ for $k\geq 1$. Given $F=\{i_1,\ldots,i_r\}\in \mathbf{C}_{n,t}$ with $1\leq i_1<\cdots<i_r\leq n$, we claim  that $$a_m\leq i_m \mbox{ for all  } m=1,\ldots, r.$$  In fact, $a_1=1\leq i_1$ and $a_2=t+1\leq i_2$. Assume that $2\leq m\leq r-1$ and $a_j\leq i_j$ for $j=1,\ldots,m$. Then, if $m=2k$ for some $k\geq 1$, then $a_{m+1}=a_m+1$ and so $i_{m+1}\geq i_m+1\geq a_m+1=a_{m+1}$; if $m=2k+1$, then $a_{m+1}-a_{m-1}=t+1$. Since $i_{m+1}-i_{m-1}\geq t+1$, it follows that $a_{m+1}\leq i_{m+1}$. This proves the claim.

We next consider the following cases:

Case 1: If $n=(t+1)p+q$ with $0\leq q\leq t-1$, then $$(p-1)t+p< n-t+1\leq (t+1)p.$$
Let $G=\{a_1,a_2,\ldots,a_{2p}\}$. Then $n\geq a_{2p}=(t+1)p\geq n+t-1$ and $a_{2p-1}=t(p-1)+p<n-t+1$, and so it is routine to verify that $G\in \mathbf{C}_{n,t}$. This proves $D\geq 2p$. To prove $D\leq 2p$, we assume on the contrary that there exists $F=\{i_1,\ldots,i_r\}\in \mathbf{C}_{n,t}$ with $1\leq i_1<\cdots<i_r\leq n$ such that $r>2p$. Then $i_{r-1}\geq i_{2p}\geq a_{2p}\geq n-t+1$. This is contradicted to condition (5). Hence, $D=2p$.

Case 2: If $n=(t+1)p+t$, then $n-t+1=(t+1)p+1=a_{2p+1},$  and one easily checks that $\{a_1,\ldots, a_{2p+1}\}$ belongs to $\mathbf{C}_{n,t}$. It follows that $D=2p+1$, by using the same argument  as in case 1. \end{proof}

  The formula for the projective dimension of $S/I_t(L_n)$ was first established in \cite[Theorem 4.1]{HV10}, and subsequently proved through an alternative method in \cite[Corollary 5.1]{BHO11}. In what follows, utilizing the theory of Alexander duality, we present a novel and concise demonstration. Recall a well-known result stating that $S/I$ is  Cohen-Macaulay if and only if the Alexander dual $I^{\vee}$ of $I$ has a linear resolution, (see for example \cite[Theorem 8.1.9]{HH}). Furthermore,  according to  \cite[Theorem 8.2.20]{HH},   $S/I$ is  sequentially Cohen-Macaulay precisely when $I^{\vee}$ is componentwise linear. Finally, for any square-free monomial ideal $I$ of $S$, we have $\pd(S/I)=\reg(I^\vee)$, as stated in \cite[Proposition~8.1.10]{HH}. Here, $\pd$ refers to the projective dimension and $\reg$ refers to the (Castelnuovo-Mumford) regularity.  For more about Alexander duality, we refer to Chapter 8 of the book \cite{HH}. .

\begin{Corollary} Let $2\leq t\leq n$ and write  $n=(t+1)p+q$ with $0\leq q\leq t$. Then $S/I_t(L_n)$ is sequentially Cohen-Macaulay and $$\mathrm{pd}(S/I_t(L_n))=\left\{
\begin{array}{ll}
2p, & \hbox{$q<t$;} \\
2p+1, & \hbox{$q=t$.}
\end{array}
\right. .$$
\end{Corollary}
\begin{proof}
  In \cite[Corollary~5.4]{Far04}, it was shown that if $\Delta$ is a simplicial forest then $S/I(\Delta)$ is   sequentially Cohen-Macaulay. Applying this result and in view of Lemma~\ref{tree}, we can conclude that $S/I_t(L_n)$ is sequentially Cohen-Macaulay.

By \cite[Theorem 8.2.20]{HH}, we know  $I_t(L_n)^{\vee}$ is componentwise linear.  Therefore, using \cite[Corollary~8.2.14]{HH}, we have $\reg(I_t(L_n)^\vee)=\Deg(I_t(L_n)^\vee)$. Furthermore,  according to \cite[Proposition~8.1.10]{HH}, we have $\pd(S/I_t(L_n))=\Deg(I_t(L_n)^\vee)$. The result now follows from Corollary~\ref{Deg}.
\end{proof}

\section{Multiplicity}

In this section, we present the main result of this paper, which provides an explicit formula for determining  the multiplicity of powers of $I_t(L_n).$   We will maintain all notation introduced in the previous sections. To start, we introduce  a general result regarding the multiplicity of square-free monomial ideals.

\begin{Lemma}\label{basic} Let $I$ be a square-free monomial ideal of height  $h$. Then $\mult(S/I)$ is equal to the number of minimal prime ideals of $I$ of  height $h$.
\end{Lemma}
\begin{proof} This follows directly from \cite[Corollary 6.2.3]{HH}.
\end{proof}

For a monomial ideal  $I$, following \cite{HM21}, we denote by $\mbox{m-grade}(I)$
the maximal length of a regular sequence of monomials contained in $I$ and call this number
the monomial grade of $I$. One has $\mbox{m-grade}(I) \leq  \mathrm{grade}(I)\leq  \height(I)$.

\begin{Lemma} \label{height}  Assume that $n=\xi t+\theta$ with $\xi \ge 1$ and  $0\le \theta<t$. Then $$\height(I_t(L_n))=\xi.$$
 \end{Lemma}
 \begin{proof} Since $u_1,u_{t+1},\ldots,u_{(\xi-1)t+1}$ forms a regular sequence that is contained in $I_t(L_n)$, it follows that $\mbox{m-grade}(I_t(L_n))\geq \xi$, and thus  $\height(I_t(L_n))\geq \xi.$ Additionally, a simple verification shows that the set $\{t,2t,\cdots,\xi t\}$ belongs to $\mathbf{C}_{n,t}$, (refer to Notation~\ref{2.6}), and so $(x_{t},\ldots, x_{\xi t})$ is a minimal prime ideal of $I_t(L_n)$. This implies $\height(I_t(L_n))\leq\xi$, completing the proof.
 \end{proof}

 In \cite{HM21}, a monomial ideal $I$ is called of {\it K$\ddot{\mathrm{o}}$nig type} if $I\neq  0$ and $\mbox{m-grade}(I) = \height(I)$. Thus, $I_t(L_n)$ is of K$\ddot{\mathrm{o}}$nig type in view of the proof of Lemma~\ref{height}.

 We next describe the minimal prime ideals of $I_t(L_n)$ of height $\xi$.

 \begin{Lemma} \label{newass} Assume $n=\xi t+\theta$ with $\xi\ge 1$ and  $0\le \theta<t$.
 Let $\p=(x_{i_1},\ldots, x_{i_\xi})$ be a prime ideal of height $\xi$ with $1\leq i_1< \cdots< i_\xi\leq n$. Then $\p\in\mathrm{ass}(I_t(L_n))$ if and only if the following conditions are satisfied:
    \begin{enumerate}
        \item[(1)] for $1\le j\le \xi$, one has $(j-1)t+\theta+1\le i_j\le jt$;
        \item[(2)] for $1\le j\le \xi-1$, one has  $i_{j+1}-i_j\le t$.
    \end{enumerate}
    \end{Lemma}

    \begin{proof} Assume that the elements of $F=\{i_1,\ldots,i_\xi\}$ meet conditions (1) and (2). We proceed to prove $F\in \mathbf{C}_{n,t}$, as  defined in   Notation~\ref{2.6}. First,  it is easy to see that $i_1\leq t$ and $i_2\geq (2-1)t+\theta+1\geq t+1$. Next, we observe that $i_{\xi-1}\leq (\xi-1)t=n-t-\theta<n-t+1$ and $i_\xi\geq (\xi-1)t+\theta+1=n-t+1$. Finally, for $j=1,\ldots, \xi-2$,  we see that
   $$i_{j+2}-i_j\geq (j+1)t+\theta+1-jt=t+\theta+1\geq t+1.$$
         Thus, $F\in \mathbf{C}_{n,t}$  and so $\p\in\mathrm{ass}(I_t(L_n))$ by Proposition~\ref{C_{n,t}}.

    Conversely, since $\p\in\mathrm{ass}(I_t(L_n))$, we have $\{i_1,\ldots,i_\xi\}$ is a minimal vertex cover of $\Delta_{n,t}$. Given that the facets  $F_1,F_{t+1},\ldots,F_{(\xi-1)t+1}$   and   $F_{\theta+1},F_{t+\theta+1},\ldots,F_{(\xi-1)t+\theta+1}$ of $\Delta_{n,t}$ are pairwise disjoint within each respective set, it follows that for $j=1,\ldots, \xi$, one has $i_j$ belongs to both $F_{(j-1)t+1}$ and $F_{(j-1)t+\theta+1}$,  and thus $(j-1)t+\theta+1\le i_j\le jt$, satisfying (1). As to condition (2), it follows  directly from Proposition~\ref{C_{n,t}}.
        \end{proof}

To determine the multiplicity of $S/I_t(L_n)$,  in light of Lemma~\ref{basic}, we must count  the number of the sequences $(i_1,\ldots, i_\xi)$ that meet conditions (1) and (2), as described in Lemma~\ref{newass}. This counting problem is purely combinatorial and will be solved in the next two results. Note that for a finite set $A$, $|A|$ denotes the cardinal number of $A$, and for a positive integer $m$, $[m]$ denotes the set $\{1,\ldots,m\}$.

    \begin{Lemma} Given $\xi\geq 1$ and $t\geq 1$, let $T_{\xi,t}$ denote the set of integral sequences $i_1<\cdots<i_\xi$ such that  \begin{enumerate}
        \item[(1)] for $1\le j\le \xi$, one has $(j-1)t+1\le i_j\le jt$;
        \item[(2)] for $1\le j\le \xi-1$, one has  $i_{j+1}-i_j\le t$.
    \end{enumerate}
    Then $$|T_{\xi,t}|=\begin{pmatrix}\xi+t-1\\\xi\end{pmatrix}.$$
    \end{Lemma}

    \begin{proof}
    We proceed by induction on  $t$. If $t=1$, the result follows easily. Assume now that  $t\geq 2$.  For each sequence $(i_1,i_2,\ldots, i_\xi)\in T_{\xi,t}$, we define $$\tau(i_1,\ldots,i_\xi):=\min\{j\in [\xi]\:\; i_j=(j-1)t+1\}.$$
    By convention,   we set $\tau(i_1,\ldots,i_\xi)=\xi+1$ if $i_j\neq (j-1)t+1$ for all $j=1,\ldots,\xi$.
    Thus, $\tau(i_1,\ldots,i_\xi)\in \{1,\ldots,\xi,\xi+1\}$. We make the following easy observation:

     {\em if $\tau(i_1,\ldots,i_\xi)=k<\xi$, then $i_{j}=(j-1)t+1$ for $j=k,\ldots,\xi$.}

     For $k\in [\xi+1]$, we define $$T_{\xi,t}^{k}=\{(i_1,i_2,\ldots, i_\xi)\in T_{\xi,t}\:\; \tau(i_1,\ldots,i_\xi)=k\}.$$

     One easily sees that $|T_{\xi,t}^1|=1$ and $|T_{\xi,t}^2|=|\{i_1: 2\leq i_1\leq t\}|=t-1$.

    In general, for all $k\in  [\xi+1]$, we use $T_k$ to denote the set of all sequences $i_1<i_2<\cdots<i_{k-1}$ such that  $(j-1)t+2\leq i_j\leq jt$ for  $j=1,\ldots,k-1,$  and $i_{j+1}-i_j\leq t$ for $j=1,\ldots,k-2.$  It follows from the above observation that   $|T_{\xi,t}^{k}|=|T_k|$ for $k=1,\ldots,\xi+1$.

     Set $i_j'=i_j-j$ for $j=1,\ldots, k-1$. Then,   we have for  $j=1,\ldots,k-1$,   $$(j-1)t+2\leq i_j\leq jt \Longleftrightarrow(j-1)(t-1)+1\leq i_j'\leq j(t-1);$$   and for $j=1,\ldots,k-2$,  $$i_{j+1}-i_j\leq t \Longleftrightarrow i'_{j+1}-i'_j\leq t-1.$$ Thus, $$(i_1,\ldots,i_{k-1})\in T_k \Longleftrightarrow (i_1',\ldots,i_{k-1}')\in T_{t-1,k-1}.$$  This implies the following map: $$(i_1,\ldots,i_{k-1})\longmapsto (i_1', \ldots, i_{k-1}') $$ establishes a bijection between $T_k$ and $T_{k-1,t-1}$. Hence, $|T_k|$ is equal to the number of $T_{k-1,t-1}$. By the inductive hypothesis, we have $|T_{k-1,t-1}|=\begin{pmatrix}k+t-3\\ k-1\end{pmatrix}.$
     Hence, it follows that
    \begin{align*}|T_{\xi,t}|
        &=\mathop{\sum}\limits_{k=1}^{\xi+1}|T^k_{\xi,t}|=\mathop{\sum}\limits_{k=1}^{\xi+1}\begin{pmatrix}k+t-3\\k-1\end{pmatrix}\\
        &=\begin{pmatrix}t-2\\0\end{pmatrix}+\begin{pmatrix}t-1\\1\end{pmatrix}+\begin{pmatrix}t\\2\end{pmatrix}+\cdots+\begin{pmatrix}\xi+t-3\\ \xi-1\end{pmatrix}+\begin{pmatrix}\xi+t-2\\ \xi\end{pmatrix}\\
        &=\begin{pmatrix}t-1\\0\end{pmatrix}+\begin{pmatrix}t-1\\1\end{pmatrix}+\begin{pmatrix}t\\2\end{pmatrix}+\cdots+\begin{pmatrix}\xi+t-3\\ \xi-1\end{pmatrix}+\begin{pmatrix}\xi+t-2\\ \xi\end{pmatrix}\\
        &=\begin{pmatrix}t\\1\end{pmatrix}+\begin{pmatrix}t\\2\end{pmatrix}+\cdots+\begin{pmatrix}\xi+t-3\\ \xi-1\end{pmatrix}+\begin{pmatrix}\xi+t-2\\ \xi\end{pmatrix}\\
        &=\begin{pmatrix}\xi+t-2\\ \xi-1\end{pmatrix}+\begin{pmatrix}\xi+t-2\\ \xi\end{pmatrix}\\
        &=\begin{pmatrix}\xi+t-1\\ \xi\end{pmatrix}.
    \end{align*}
 Here, we use the formula $\binom{n}{k}+\binom{n}{k+1}=\binom{n+1}{k+1}$ many times.
    \end{proof}
   \begin{Corollary} \label{number2} Given $n\geq t\geq 1$ and assume $n=\xi t+\theta$  with $0\leq \theta \leq t-1$.  Let $X_{n,t}$ denote the set of integral sequences $i_1<\ldots<i_a$ such that  \begin{enumerate}
        \item[(1)] for $1\le j\le \xi$, one has $(j-1)t+\theta+1\le i_j\le jt$;
        \item[(2)] for $1\le j\le \xi-1$, one has  $i_{j+1}-i_j\le t$.
    \end{enumerate}
    Then $$|X_{n,t}|=\begin{pmatrix}\xi+t-\theta-1\\ \xi\end{pmatrix}.$$
    \end{Corollary}
    \begin{proof} Set $i_j'=i_j-\theta j$. Then, we have  for $1\le j\le \xi$,
    $$ (j-1)t+\theta+1\le i_j\le jt \Longleftrightarrow  (j-1)(t-\theta)+1\leq  i_j'\leq j(t-\theta),$$
    and for $1\le j\le \xi-1$,
    $$i_{j+1}-i_j\le t \Longleftrightarrow i_{j+1}'-i_j'\le t-\theta.$$
    Thus, the map: $(i_1,\ldots, i_\xi)\longmapsto (i_1',\ldots, i_\xi')$ establishes a bijection between
$X_{n,t}$ and $T_{\xi,t-\theta}$.    This implies $$|X_{n,t}|=|T_{\xi,t-\theta}|=\begin{pmatrix}\xi+t-\theta-1\\ \xi\end{pmatrix},$$ as desired.
    \end{proof}

 \begin{Proposition} \label{squarefree-mult} Assume $n=\xi t+\theta$ with $\xi\ge 1$ and $0\leq \theta<t$. Then  $$\mult(S/I_t(L_n))=\begin{pmatrix}\xi+t-\theta-1\\ \xi\end{pmatrix}.$$
\end{Proposition}
\begin{proof} This  is a combination of Lemma~\ref{basic}, Lemma~\ref{newass}, and Corollary~\ref{number2}.
\end{proof}

Based on Proposition~\ref{squarefree-mult}, we can obtain our main result through an induction process. In this process, constructing suitable short exact sequences is the crucial step. In order to achieve this, we require the following lemmas.

Recall the {\it support} of monomial $u$ is the set $\{i\in [n]\:\; x_i \mbox{ divides } u\}$. For a monomial ideal $I$, we define the {\it support} of $I$ to be the union of supports of minimal generators of $I$. If monomial ideals $J$ and $K$ have disjoint supports, then $J:K=J$. Thus, if $I$ is another monomial ideal,  then we have $(I+J):K=(I:K)+(J:K)=(I:K)+J.$   This fact is utilized  in the following proofs without explicit reference.

\begin{Lemma} \label{multpower1}
 For all $s\geq 2$, we have
 \begin{enumerate}
   \item $I_t(L_n)^s:u_{n-t+1}=I_t(L_n)^{s-1}$;
   \vspace{2mm}

   \item Let $(I_t(L_{n-1})^s:u_{n-t+1}, x_{n-t})$ denote the ideal $(I_t(L_{n-1})^s:u_{n-t+1})+ (x_{n-t})$. Then $$(I_t(L_{n-1})^s:u_{n-t+1}, x_{n-t})=(I_t(L_{n-t-1})^s, x_{n-t}).$$ In particular, if $n\leq 2t$, then $I_t(L_{n-t-1})^s=0$ and so $(I_t(L_{n-1})^s:u_{n-t+1}, x_{n-t})=(x_{n-t}).$
 \end{enumerate}
              \end{Lemma}
   \begin{proof} (1) is just \cite[Lemma 2.1]{BC23}.

   (2) Since $x_{n-t}$ does not divide $u_{n-t+1}$, it follows that $$(I_t(L_{n-1})^s:u_{n-t+1})+ (x_{n-t})=(I_t(L_{n-1})^s, x_{n-t}):u_{n-t+1}.$$
   It is easy to see that $(I_t(L_{n-1})^s, x_{n-t})=(I_t(L_{n-t-1})^s, x_{n-t}).$  Note that  $(I_t(L_{n-t-1})^s, x_{n-t})$ and $u_{n-t+1}$ have disjoint supports, the result follows.
   \end{proof}

        \begin{Lemma} \label{multpower2}
 Let $A_i:=(I_t(L_{n-i})^s,x_{n-t+1}\cdots x_{n-i+1})$ and let $B_i:=(A_i,x_{n-i+1})$ for $i=1,\ldots, t$. Then $B_i=(I_t(L_{n-i})^s, x_{n-i+1})$ and $A_i:x_{n-i+1}=A_{i+1}$ for $i=1,\ldots, t$.
\end{Lemma}

\begin{proof}  The equality that $B_i=(I_t(L_{n-i})^s, x_{n-i+1})$ is easy.
Let us show that $A_i:x_{n-i+1}=A_{i+1}$. Since $$I_t(L_{n-i})^s=(I_t(L_{n-i-1})^s,u_{n-i-t+1}I_t(L_{n-i})^{s-1}),$$ we have
    \begin{align*}
      &A_i:x_{n-i+1}=(I_t(L_{n-i})^s,x_{n-t+1}\cdots x_{n-i+1}):x_{n-i+1}\\
      &=(I_t(L_{n-i-1})^s,u_{n-i-t+1}I_t(L_{n-i})^{s-1}:x_{n-i+1})+(x_{n-t+1}\cdots x_{n-i})\\
      &=(I_t(L_{n-i-1})^s,u_{n-i-t+1}I_t(L_{n-i})^{s-1}, x_{n-t+1}\cdots x_{n-i}):x_{n-i+1}\\
      &=(I_t(L_{n-i-1})^s, x_{n-t+1}\cdots x_{n-i}):x_{n-i+1}=A_{i+1}.
    \end{align*}
Here, the third equality follows because $ x_{n-t+1}\cdots x_{n-i}$ divides $u_{n-i-t+1}$.
   \end{proof}

    Let $T$ be a finitely generated graded $S$-modules of dimension $d$. According to  \cite[Corollary 4.1.8]{BH98},   we can express the Hilbert series of $T$ as  $$\HS(T,z)=\frac{Q(T,z)}{(1-z)^d}$$ for some $Q(T,z)\in \mathbb{Q}[z,z^{-1}]$ with $Q(T,1)\neq 0$. Furthermore, \cite[Proposition 4.1.9]{BH98} asserts $$\mult(T)=Q(T,1).$$  Based on these results, we can promptly derive the subsequent lemma.

   \begin{Lemma} \label{mult} Let $0\longrightarrow K\longrightarrow M\longrightarrow N\longrightarrow 0$  be a short exact sequence of  finitely generated graded $S$-modules. Then $\dim M=\max\{\dim K, \dim N\}$. Moreover,
   \begin{enumerate} \item If $\dim M=\dim K=\dim N$, then $\mult (M)=\mult (N)+\mult (K)$;
   \item If $\dim M=\dim K>\dim N$, then $\mult (M)=\mult (K)$;
   \item If $\dim M=\dim N>\dim K$, then $\mult (M)=\mult (N)$.
   \end{enumerate}
   \end{Lemma}
We  also need  the following fact: given that $I$ is a monomial ideal of $S$ and $y$ is a new variable, then $$\mult (S/I)=\mult (S[y]/IS[y]).$$

We  are  now ready to present the main result of this paper.
   \begin{Theorem} \label{main}
     Assume $n=\xi t+\theta$, with $\xi\ge 1$ and $0\le \theta<t$. Then, for all $s\ge 1$, we have $$\mult (S/I_t(L_n)^s)=\begin{pmatrix}s+\xi-1\\s-1\end{pmatrix}\begin{pmatrix}\xi+t-\theta-1\\ \xi\end{pmatrix}.$$
\end{Theorem}
\begin{proof} We proceed by induction on both $s$ and $n$.
     The case that $s=1$ follows from Proposition~\ref{squarefree-mult}.  If $n=t$, we have $$\HS(S/I_t(L_t)^s,z)=\frac{1-z^{st}}{(1-z)^t}=\frac{1}{(1-z)^{t-1}}\sum_{i=0}^{st-1}z^i.$$  It follows that $Q(S/I_t(L_t)^s, z)=\sum_{i=0}^{st-1}z^i$ and so $\mult(S/I_t(L_t)^s)=Q(S/I_t(L_t)^s,1)= st.$  This proves the result holds  when $n=t$. Now, assume $s>1$ and  $n>t$.  To begin, we write $n=\xi t+\theta$, where $\xi\ge 1$ and $0\le \theta<t$.
     There are two cases to consider.

Case  1:   $\theta> 0$. Consider the following three short exact sequences of graded $S$-modules:
     \begin{equation}\label{first} 0 \To\frac{S}{I_t(L_n)^s:u_{n-t+1}}[-t] \To\frac{S}{I_t(L_n)^s} \To\frac{S}{(I_t(L_n)^s,u_{n-t+1})} \To 0,\end{equation}
    \begin{align*}
    0 \To\frac{S}{I_t(L_{n-1})^{s}:u_{n-t+1}}[-t] \To\frac{S}{I_t(L_{n-1})^s} \To\frac{S}{(I_t(L_{n-1})^s,u_{n-t+1})} \To 0,
    \end{align*}
    and
    \begin{align*}0 \To\frac{S}{I_t(L_{n-1})^{s}:x_{n-t} u_{n-t+1}}[-1] \To\frac{S}{I_t(L_{n-1})^s:u_{n-t+1}}\\ \To\frac{S}{(I_t(L_{n-1})^s:u_{n-t+1},x_{n-t})} \To 0.\end{align*}

    We proceed to show that all nonzero graded modules appearing in the above three short exact sequences  share the same dimension, which is $(n-\xi)$. Firstly,  by Lemma~\ref{multpower1}, we have $I_t(L_n)^{s}:u_{n-t+1}=I_t(L_n)^{s-1}$,
      it follows from Lemma~\ref{height} that $$\dim \frac{S}{I_t(L_n)^{s}:u_{n-t+1}}=\dim \frac{S}{I_t(L_n)^{s}}=n-\xi.$$ Note that $(I_t(L_n)^{s},u_{n-t+1})=(I_t(L_{n-1})^{s},u_{n-t+1})$ and $\sqrt{(I_t(L_{n-1})^{s},u_{n-t+1})}=I_t(L_n)$,  we obtain  $$\dim \frac{S}{(I_t(L_n)^{s},u_{n-t+1})}=\dim \frac{S}{(I_t(L_{n-1})^{s},u_{n-t+1})}=n-\xi.$$
Since $$I_t(L_{n-1})^s:x_{n-t}u_{n-t+1}=I_t(L_{n-1})^s:u_{n-t}x_n=I_t(L_{n-1})^{s-1}:x_n=I_t(L_{n-1})^{s-1},$$ we have
      $$\dim \frac{S}{I_t(L_{n-1})^s:x_{n-t}u_{n-t+1}}=\dim \frac{S}{I_t(L_{n-1})^{s-1}}=n-\xi.$$
      
     If $n>2t$, by Lemma~\ref{multpower1}.(2), we have $(I_t(L_{n-1})^s:u_{n-t+1},x_{n-t})=(I_t(L_{n-t-1})^s,x_{n-t})$. Since $n-t-1=(\xi-1)t+\theta-1$ with $\theta-1\geq 0$,  it follows that $$\dim \frac{S}{(I_t(L_{n-1})^s:u_{n-t+1},x_{n-t})}=\dim \frac{S}{(I_t(L_{n-t-1})^s,x_{n-t})}=n-1-(\xi-1)=n-\xi.$$

 By Lemma~\ref{mult} and using the induction hypothesis, it follows that
    \begin{align*}
      & \mult (S/I_t(L_n)^s)= \mult (S/I_t(L_n)^{s-1})+\mult (S/I_t(L_{n-1})^s)
      \\&\qquad \qquad\qquad \qquad -\mult (S/I_t(L_{n-1})^{s-1})-\mult (S/I_t(L_{n-t-1})^s) \\
            =&\begin{pmatrix}s+\xi-2\\s-2\end{pmatrix}\begin{pmatrix}\xi+t-\theta-1\\ \xi\end{pmatrix}
      +\begin{pmatrix}s+\xi-1\\s-1\end{pmatrix}\begin{pmatrix}\xi+t-\theta\\ \xi\end{pmatrix}\\
      &~~~~-\begin{pmatrix}s+\xi-2\\s-2\end{pmatrix}\begin{pmatrix}\xi+t-\theta\\ \xi\end{pmatrix}
      -\begin{pmatrix}s+\xi-2\\s-1\end{pmatrix}\begin{pmatrix}\xi+t-\theta-1\\\xi-1\end{pmatrix}\\
            =&\begin{pmatrix}s+\xi-1\\s-1\end{pmatrix}\begin{pmatrix}\xi+t-\theta\\ \xi\end{pmatrix}-\begin{pmatrix}s+\xi-2\\s-2\end{pmatrix} \begin{pmatrix}\xi+t-\theta-1\\ \xi-1\end{pmatrix}
     \\ & -\begin{pmatrix}s+\xi-2\\s-1\end{pmatrix}\begin{pmatrix}\xi+t-\theta-1\\ \xi-1\end{pmatrix}\\
            =&\begin{pmatrix}s+\xi-1\\s-1\end{pmatrix}\begin{pmatrix}\xi+t-\theta\\ \xi\end{pmatrix}-\begin{pmatrix}s+\xi-1\\s-1\end{pmatrix}\begin{pmatrix}\xi+t-\theta-1\\ \xi-1\end{pmatrix}\\
      =&\begin{pmatrix}s+\xi-1\\s-1\end{pmatrix}\begin{pmatrix}\xi+t-\theta-1\\ \xi\end{pmatrix}.
    \end{align*}

     If $t<n\leq 2t$, then $\xi=1$ and so $$\dim \frac{S}{(I_t(L_{n-1})^s:u_{n-t+1},x_{n-t})}=\dim \frac{S}{(x_{n-t})}=n-\xi.$$
Since $\mult(S/(x_{n-t}))=1$, by applying the induction hypothesis together with Lemma~\ref{mult}, we can deduce that:
    \begin{align*}
      &\mult(S/I_t(L_n)^s)\\
      =&\mult(S/I_t(L_n)^{s-1})+\mult(S/I_t(L_{n-1})^s)-\mult(S/I_t(L_{n-1})^{s-1})-1\\
      %µÚ2žöµÈºÅ
      =&\begin{pmatrix}s-1\\s-2\end{pmatrix}\begin{pmatrix}t-\theta \\1\end{pmatrix}
      +\begin{pmatrix}s\\s-1\end{pmatrix}\begin{pmatrix}t-\theta+1\\1\end{pmatrix}-\begin{pmatrix}s-1\\s-2\end{pmatrix}\begin{pmatrix}t-\theta+1\\1\end{pmatrix}-1\\
      %µÚ3žöµÈºÅ
      =&\begin{pmatrix}s\\s-1\end{pmatrix}\begin{pmatrix}t-\theta+1\\1\end{pmatrix}-\begin{pmatrix}s-1\\s-2\end{pmatrix} \begin{pmatrix}t-\theta\\0\end{pmatrix}-\begin{pmatrix}s-1\\s-1\end{pmatrix} \begin{pmatrix}t-\theta\\0\end{pmatrix}\\
      %µÚ4žöµÈºÅ
      =&\begin{pmatrix}s\\s-1\end{pmatrix}\begin{pmatrix}t-\theta+1\\1\end{pmatrix}-\begin{pmatrix}s\\s-1\end{pmatrix} \begin{pmatrix}t-\theta\\0\end{pmatrix}\\
      =&\begin{pmatrix}s\\s-1\end{pmatrix}\begin{pmatrix}t-\theta\\1\end{pmatrix}.
    \end{align*}

 \noindent Case 2:   $\theta=0$, i.e., $n=\xi t$. We may assume $\xi>1$, since the case that $\xi=1$  has been addressed. As in Lemma~\ref{multpower2},  we put for $i=1,\ldots, t$, $$A_i:=(I_t(L_{n-i})^s,x_{n-t+1}\cdots x_{n-i+1}) \mbox{\quad and \quad} B_i:=(A_i,x_{n-i+1}).$$  We must compute the dimensions of  these graded $S$-modules.   Fix $i\in \{1,\ldots,t\}$.  By Lemma~\ref{multpower2}, we see that $B_i=(I_t(L_{n-i})^s,x_{n-i+1})$. Since  we $n-i=(\xi-1)t+t-i$ with $0\leq t-i<t$, it follows from Lemma~\ref{height} that $\height(I_t(L_{n-i})^s)=\xi-1$ and so   $\height(B_i)=\height(I_t(L_{n-i})^s)+1=\xi$.  We next show that $\height(A_i)=\xi$. Since $\height (I_t(L_{n-i})^s)=\xi-1$, it follows that $\height(A_i)$ is either $\xi-1$ or $\xi$. Suppose on the contrary that $\height(A_i)=\xi-1$. Then there must exist a minimal prime ideal, say $\p$,  of $I_t(L_{n-i})^s$ with $\height(\p)=\xi-1$ that contains $A_i$. However, this can not be the case. To see why, we assume without loss the generality that  $\p=(x_{i_1},\ldots, x_{i_{\xi-1}})$ with $1\leq i_1< \cdots< i_{\xi-1}\leq n-i$. By Lemma~\ref{newass}, we have $(\xi-2)t+t-i+1\le i_{\xi-1}\le (\xi-1)t$.  This implies $A_i\nsubseteq \p$, leading to contradiction. Therefore, we conclude that $\height(A_i)=\xi$.  As a result, we have $$\dim\frac{S}{A_i}=\dim\frac{S}{B_i}=n-\xi \mbox{ for all } i=1,\ldots,t.$$

 Since $A_1=(I_t(L_{n-1})^{s}, u_{n-t+1})$, according to  Lemma~\ref{mult} and considering the short exact sequence (1), it follows that \begin{equation}\label{second}\mult(S/I_t(L_n)^s)=\mult(S/I_t(L_n)^{s-1})+\mult(S/A_1).\end{equation}

 On the other hand,  Lemma~\ref{multpower2} asserts the existence of the following $t-1$ short exact sequences:
     \begin{align*}
       & 0 \To\frac{S}{A_{2}}[-1] \To\frac{S}{A_1} \To\frac{S}{B_1} \To 0,\\
       & 0 \To\frac{S}{A_{3}}[-1] \To\frac{S}{A_2} \To\frac{S}{B_2} \To 0,\\
        & \qquad\qquad\qquad\qquad  \vdots\\
         &   0 \To\frac{S}{A_{t}}[-1] \To\frac{S}{A_{t-1}} \To\frac{S}{B_{t-1}} \To 0.
    \end{align*}
 In view of these short exact sequences together with (\ref{second}), we obtain
    \begin{align*}
      &\mult(S/I_t(L_n)^s)=\mult(S/I_t(L_n)^{s-1})+\mult(S/A_1)\\
      =&\mult(S/I_t(L_n)^{s-1})+\mult(S/A_2)+\mult(S/B_1)\\
      =&\mult(S/I_t(L_n)^{s-1})+\mult(S/A_3)+\mult(S/B_1)+\mult(S/B_2)\\
      =&\cdots\\
      =&\mult(S/I_t(L_n)^{s-1})+\mult(S/A_t)+\mathop{\sum}\limits_{i=1}^{t-1} \mult(S/B_i).
               \end{align*}
                       Since $A_t=B_t$, it follows that
        \begin{equation*} \mult(S/I_t(L_n)^s)=\mult(S/I_t(L_n)^{s-1})+\mathop{\sum}\limits_{i=1}^{t} \mult(S/B_i)\end{equation*}
          This implies
       \begin{align*}
      &\mult(S/I_t(L_n)^s)=\mult(S/I_t(L_n)^{s-1})+\mathop{\sum}\limits_{i=1}^t \mult(S/I_t(L_{n-i})^s)\\
      %µÚ2žöµÈºÅ
      =&\begin{pmatrix}s+\xi-2\\s-2\end{pmatrix}\begin{pmatrix}\xi+t-1\\ \xi\end{pmatrix}
      +\begin{pmatrix}s+\xi-2\\s-1\end{pmatrix} \mathop{\sum}\limits_{i=1}^t\begin{pmatrix}\xi+i-2\\ \xi-1\end{pmatrix}\\
      %µÚ3žöµÈºÅ
      =&\begin{pmatrix}s+\xi-2\\s-2\end{pmatrix}\begin{pmatrix}\xi+t-1\\ \xi\end{pmatrix}
      +\begin{pmatrix}s+\xi-2\\s-1\end{pmatrix}\begin{pmatrix}\xi+t-1\\ \xi\end{pmatrix}\\
      =&\begin{pmatrix}s+\xi-1\\s-1\end{pmatrix}\begin{pmatrix}\xi+t-1\\ \xi\end{pmatrix}.
    \end{align*} This completes the proof.
\end{proof}

\begin{Remark} \em According to \cite[Theorem 1.1]{HPV08}, $\mult (S/I_t(L_n)^s)$ is of  polynomial type of degree $\leq \xi$ in $s$. However, Theorem~\ref{main} reveals that  $\mult (S/I_t(L_n)^s)$ is a polynomial of degree $\xi$ in $s$. Consequently, Theorem~\ref{main} serves a valuable complement to \cite[Theorem 1.1]{HPV08}.
\end{Remark}

{\bf \noindent Acknowledgment:}
This research is supported by NSFC (No. 11971338).  We would express our sincere gratitude to the referee for his/her careful reading and numberous comments that improve the presentation of our paper greatly.

\end{document}